\documentclass[10pt]{amsart}

\usepackage{amsmath,amsfonts,amssymb,amsthm}

\usepackage{pstricks}

\newtheorem{theorem}{Theorem}[section]
\newtheorem{lemma}[theorem]{Lemma}
\newtheorem{proposition}[theorem]{Proposition}
\newtheorem{corollary}[theorem]{Corollary}

\theoremstyle{definition}
\newtheorem{definition}[theorem]{Definition}

\newtheorem{problem}[theorem]{Problem}

\theoremstyle{remark}
\newtheorem{remark}[theorem]{Remark}

\numberwithin{equation}{section}

\newcommand{\qc}[2]{\left[{#1};{#2}\right]_q}                                

\newcommand{\am}[1]{a^{-}(#1)}
\newcommand{\ap}[1]{a^{+}(#1)}
\newcommand{\az}[1]{a^{0}(#1)}

\newcommand{\amx}[1]{a_X^{-}(#1)}
\newcommand{\apx}[1]{a_X^{+}(#1)}
\newcommand{\azx}[1]{a_X^{0}(#1)}

\newcommand{\iloskal}[2]{\left<{#1},{#2}\right>}                                
\newcommand{\is}[1]{\left<{#1}\right>}                                
\newcommand{\pstwo}{\mathbb{P}}                                                 
\newcommand{\wo}{\mathbb{E}}                                                    

\newcommand{\FFF}{\textrm{\textsf{F}}}
\newcommand{\NC}{\textrm{\textsf{NC}}}

\newcommand{\nat}{\mathbb{N}}                                                   
\newcommand{\rzecz}{\mathbb{R}}                                                 
\newcommand{\zesp}{\mathbb{C}}                                                  

\begin{document}

\title[Moments and $q$-commutators]{Moments and $q$-commutators of noncommutative random vectors}

\author[K.Dro\.{z}d\.{z}ewicz]{Krzysztof Dro\.{z}d\.{z}ewicz}
\address{Wydzia{\l} Matematyki i Nauk Informacyjnych\\
Politechnika Warszawska\\
Pl. Politechniki 1\\
00-661 Warsaw, Poland} \email{K.Drozdzewicz@mini.pw.edu.pl}

\author[W.Matysiak]{Wojciech Matysiak}
\address{Wydzia{\l} Matematyki i Nauk Informacyjnych\\
Politechnika Warszawska\\
Pl. Politechniki 1\\
00-661 Warsaw, Poland} \email{matysiak@mini.pw.edu.pl}

\thanks{}

\subjclass[2000]{Primary: 46L53. Secondary: 81S05, 62H05.}

\keywords{$q$-commutator, annihilation operator, creation operator, preservation operator, moments}



\begin{abstract}
A method for computing the mixed moments of (not necessarily commutative) random vectors from the first order moments, the $q$-commutators between the annihilation and creation operators, and the $q$-commutators between the annihilation and preservation operators, is presented. The method is illustrated by a relevant characterization of $q$-Gaussian vectors.
\end{abstract}

\maketitle

\section{Introduction}\label{s:intro}

There has appeared recently a number of papers on the problem of identifying probability measures via the annihilation, preservation and creation operators associated with them. Building on previous works of Accardi, Kuo and Stan (see e.g. \cite{accakuostan1},\cite{accakuostan2}), Stan and Whitaker \cite{stanwhitaker} showed that the moments of a probability measure on $\rzecz^d$ can be recovered from two families of commutators, namely the commutators between the annihilation and creation operators, and the commutators between the annihilation and preservation operators, provided that the moments of the first order are known. The authors presented a concrete method for recovering the moments and applied it to identify one-dimensional normal distribution and one-dimensional Meixner distributions, postponing any multidimensional or noncommutative computations to some subsequent papers. In \cite{stanmeixner}, the foundations for defining the joint annihilation-preservation-creation decomposition of a finite family of not necessarily commutative random variables were laid.
The analysis of the key example from \cite{stanmeixner}, the Meixner random vector $(X,Y)$ of class $\mathcal{M}_L$, partially rests on a noncommutative extension (stated without a proof) of the method developed in \cite{stanwhitaker}. (For some recent advances on Meixner classes of non-commutative generalized  stochastic processes consult \cite{bozejkolytvynov1} and \cite{bozejkolytvynov2}.)

Our aim in this note is to provide more insight on the method when applied in dimension greater than 1. 
We present a direct multidimensional application of a slight generalization of the method from \cite{stanwhitaker}. The generalization is twofold: we allow the analyzed random variables to be not necessarily commutative and we use $q$-commutators
\[
[A;B]_q:=AB-q\cdot BA
\]
($A$ and $B$ are some operators, $q\in[-1,1]$). Next we show concretely how the generalized method works by solving a $q$-commutator problem connected with $q$-Gaussian random variables defined by Bo\.{z}ejko and Speicher in \cite{bosp}, where a remarkable $q$ version of the Fock space was introduced. As special cases, we get some characterizations of the multivariate normal distribution and of semicircular families.

\section{Noncommutative probability background}

\subsection{A joint annihilation-preservation-creation decomposition}

We follow Stan \cite{stanmeixner} when presenting the material from this Section. Let $\left(H,\iloskal{\cdot}{\cdot}\right)$ be a Hilbert space over $\rzecz$ and let $X_1,X_2,\ldots,X_d$ be symmetric, densely defined linear operators on $H$. Denote by $\mathcal{A}$ the unital algebra generated by $X_1,\ldots,X_d$ (unital means that it is assumed that $I\in\mathcal{A}$ with $I$ being the identity operator on $H$). We fix a unit vector $\phi\in H$ belonging to the domain of any $g$ from $\mathcal{A}$ (we assume its existence) and call it the \emph{vacuum vector}. The elements of $\mathcal{A}$ will be \emph{random variables} and the pair $(\mathcal{A},\phi)$ will be a \emph{probability space supported by} $H$. For any $g\in\mathcal{A}$, we will call $\wo[g]:=\iloskal{g\phi}{\phi}$ the \emph{expectation} of the random variable $g$.

One can show that for any $g\in\mathcal{A}$ there exists a polynomial $p$ in $d$ noncommutative variables such that $g=p(X_1,\ldots,X_d)$. After Stan \cite{stanmeixner} we introduce the following equivalence relation:
\begin{definition}
Let $(\mathcal{A},\phi)$ and $(\mathcal{A}^\prime,\phi^\prime)$ be two probability spaces (with expectations $\wo$ and $\wo^\prime$), supported by two Hilbert spaces $H$ and $H^\prime$. We say that random vectors $(X_1,\ldots,X_d)$ with $X_i\in\mathcal{A}$ and $(X_1^\prime,\ldots,X_d^\prime)$ with $X_i^\prime\in\mathcal{A}^\prime$ are \emph{moment equal} if for any polynomial $p$ in $d$ noncommutative variables
\[
\wo\left[p(X_1,\ldots,X_d)\right]=\wo^\prime\left[p(X_1^\prime,\ldots,X_d^\prime)\right].
\]
\end{definition}

For $n\in\nat\cup\{0\}$, define
\[
F_n=\{p(X_1,\ldots,X_d)\phi: p \textrm{ - a polynomial of total degree less than or equal to }n\}.
\]
Clearly, $F_n$ is a finite dimensional (thus closed) subspace of $H$ and $F_0\subset F_1\subset\ldots\subset H$. Put $G_{-1}:=\{0\}$, $G_0:=F_0$ and as $G_n$ take the orthogonal complement of $F_{n-1}$ into $F_{n}$ for $n\in\nat$. Each $G_n$ is the \emph{homogenous chaos space of order} $n$ generated by $X_1,\ldots,X_d$. Similarly,
\[
\mathcal{H}:=\bigoplus_{n=0}^\infty G_n
\]
is called  \emph{the chaos space generated by} $X_1,\ldots,X_d$. It is easily seen that $\mathcal{H}$ is the closure of $\mathcal{A}\phi:=\{g\phi:g\in\mathcal{A}\}$ in $H$.

Since $X_1,\ldots,X_d$ are symmetric operators, it is not difficult to prove (see \cite{accakuostan1}) that $X_i G_n\perp G_k$ for any $i\in\{1,\ldots,d\}$, $n\in\nat\cup\{0\}$ and non-negative integer $k\notin\{n-1,n,n+1\}$. It implies for any $i\in\{1,\ldots,d\}$ the existence of three operators $\am{i}$, $\az{i}$ and $\ap{i}$ such that
\begin{equation}\label{xexpans}
X_i=\am{i}+\az{i}+\ap{i}
\end{equation}
(the domain of the operators in \eqref{xexpans} is $\mathcal{A}\phi$). We will call them the \emph{annihilation}, \emph{preservation} and \emph{creation} operators, respectively, what agrees with the fact that for $i\in\{1,\ldots,d\}$ and $n\in\nat\cup\{0\}$,
\[
\am{i}(G_n)\subset G_{n-1},\ \az{i}(G_n)\subset G_n\textrm{ and } \ap{i}(G_n)\subset G_{n+1}.
\]

For the future reference, we now list some easily derived properties of the introduced operators. We shall be using the symmetry of $\az{i}$ and the duality between $\ap{i}$'s and $\am{i}$'s, particularly the identities:
\begin{equation}\label{sprzez}
\iloskal{a^0(i)\phi}{\phi}=\iloskal{\phi}{a^0(i)\phi},\quad  \iloskal{a^+(i)\phi}{\phi}=\iloskal{\phi}{a^-(i)\phi}, \quad i=1,\ldots,d.
\end{equation}
We also will need the following consequences of the fact that $\phi\in G_0$:
\begin{equation}\label{a0}
a^-(i)\phi=0\ \textrm{ and }\ a^0(i)\phi=\wo[X_i]\phi, \quad i=1,\ldots,d.
\end{equation}
Similarly, it is easily observed that if $n\in\nat$ and $\sigma:\{1,\ldots,n\}\to\{1,\ldots,d\}$, then
\begin{equation}\label{moreminuses}
\am{\sigma(1)}\ldots \am{\sigma(k)} X_{\sigma(k+1)}\ldots X_{\sigma(n)}\phi=0
\end{equation}
for $k>\lfloor n/2\rfloor$.

\begin{remark}\label{r:przemienny}
If $X_1,\ldots X_d$ are classical (commutative) random variables defined on the same probability space $(\Omega,\mathcal{F},\pstwo)$ and having finite moments of any order, one can consider $H=L^2(\Omega,\mathcal{F},\pstwo)$ and $\phi\equiv 1$ (the constant random variable, equal to 1), and regard $X_1,\ldots,X_d$ as multiplication operators on the space $\mathcal{A}1\subset H$. Obviously, $X_i X_j=X_j X_i$ for all $1\le i\le j\le d$.
\end{remark}

\subsection{Catalan sequences}
If $n\in\nat$, $\varepsilon=(\varepsilon(1),\ldots,\varepsilon(2n))\in\{1,-1\}^{2n}$ and $k\in\{1,\ldots,2n\}$, set $\tau_{k}=\sum_{j=k}^{2n}\varepsilon(j)$. We say that $\varepsilon$ is a \emph{Catalan sequence} if $\tau_{2n}>0$, $\tau_{2n-1}\ge0,\ldots,\tau_{2}\ge0$, $\tau_1=0$ and let $C(2n)$ be the set of such sequences.

The Catalan sequences are important in the context of this note for the following reason. If $X_1,\ldots,X_{2n}$ are random variables from a probability space $\left(\mathcal{A},\phi\right)$ and $\varepsilon=(\varepsilon(1),\ldots,\varepsilon(2n))\in\{1,-1\}^{2n}$, then it is easily seen that
\begin{equation*}
\iloskal{a^{\varepsilon(1)}(1)\ldots a^{\varepsilon(2n)}(2n)\phi}{\phi}=0
\end{equation*}
unless $\varepsilon\in C(2n)$ (here we adopt the convention that $a^{+1}(k)=a^{+}(k)$ and $a^{-1}(k)=a^{-}(k)$ for any $k$ , thus $a^{\pm 1}(k)$ are connected with $X_k$ via \eqref{xexpans}).

\subsection{$q$-Fock spaces and $q$-Gaussian random variables}

We briefly outline here the basics of Bo\.{z}ejko and Speicher \cite{bosp} theory of $q$-Fock spaces. For a real Hilbert space $H_0$ with complexification $H$, its $q$-Fock space $\Gamma_q(H)$ is defined as the closure of the algebraic sum $\zesp\phi\oplus\bigoplus_n H^{\otimes n}$ in the scalar product
\[
\iloskal{f_1\otimes\ldots\otimes f_n}{g_1 \otimes\ldots\otimes g_m}_q=\delta_{m,n}\sum_{\pi\in S_n} q^{|\pi|}\cdot\prod_{j=1}^n \iloskal{f_j}{g_{\pi(j)}},
\]
where $\phi$ is the vacuum vector, $S_n$ are permutations of $\{1,\ldots,n\}$ and
\[
|\pi|=\#\{(i,j):i<j,\pi(i)>\pi(j)\}.
\]
The range of $q$ is $[-1,1]$ -- if $q=-1$ or $q=1$ one must first divide out by the null space, and then one obtains the usual antisymmetric (Fermion) and symmetric (Boson) Fock space. For $q=0$ the construction gives the full (Voiculescu) Fock space.

Now, for $f\in H$ we define the creation operator $a^+(f)$, and the annihilation operator $a^-(f)$ on $\zesp\phi\oplus\bigoplus_n H^{\otimes n}$ in the following way:
\begin{eqnarray*}
a^+(f)\phi &=& f,\\
a^+(f) f_1\otimes\ldots\otimes f_n &=& f\otimes f_1\otimes\ldots\otimes f_n,
\end{eqnarray*}
and
\begin{eqnarray*}
a^-(f)\phi &=& 0,\\
a^-(f) f_1\otimes\ldots\otimes f_n &=& \sum_{j=1}^n q^{j-1}\cdot\iloskal{f}{f_j}\cdot f_1\otimes\ldots f_{j-1}\otimes f_{j+1}\otimes\ldots\otimes f_n.
\end{eqnarray*}
If $q<1$ then $a^+(f)$ and $a^-(f)$ extend to bounded operators on $\Gamma_q(H)$ and $a^-(f)$ is $\iloskal{\cdot}{\cdot}_q$-adjoint of $a^+(f)$. They satisfy the commutation relation $\qc{a^-(f)}{a^+(g)}=\iloskal{f}{g}I$.

The self-adjoint operators $x(f)=a^+(f)+a^-(f)$ are called $q$-Gaussians. 

\section{The main result}

\subsection{A characterization}

Having introduced the joint annihilation-pre\-ser\-va\-tion-creation decomposition, we are in a position to state the main result of the paper.

\begin{theorem}\label{t:characterization}
Let $\left(\mathcal{A},\phi\right)$ and $\left(\mathcal{B},\psi\right)$ be probability spaces with the expectations $\wo$ and $\wo^\prime$, supported by Hilbert spaces
$H$ and $H^\prime$, respectively. Random vectors $X=(X_1,\ldots,X_d)$ (with $X_i\in\mathcal{A}$) and $Y=(Y_1,\ldots,Y_d)$ (with $Y_i\in\mathcal{B}$) are moment equal if and only if the following three conditions hold for all $i,j\in\{1,\ldots,d\}$:
\begin{eqnarray}
\qc{a_X^-(i)}{a_X^+(j)}&=&\qc{a_Y^-(i)}{a_Y^+(j)},\label{anikrea}\\
\qc{a_X^-(i)}{a_X^0(j)}&=&\qc{a_Y^-(i)}{a_Y^0(j)},\label{anizacho}\\
\wo[X_i]&=&\wo^\prime[Y_i]\label{equimom}.
\end{eqnarray}
(Here and later, the symbols $a_U^-(i)$, $a_U^0(i)$ and $a_U^+(i)$ will stand for the annihilation, preservation and creation operators associated with the $i$-th coordinate of a random vector $U$ via \eqref{xexpans}.)
\end{theorem}

Theorem \ref{t:characterization} generalizes the main result of \cite{stanwhitaker} to $q$-commutators (note that there are no assumptions imposed on $q$ -- it can be an arbitrary real or complex number) and to not necessarily commutative random variables.

\subsection{Proof of the main result}

\subsubsection{Notation}
Here we present some notation to be used in the remainder of the paper.

While working with mixed moments of the form
\begin{equation}\label{moment}
\wo\left[X_{\sigma(1)}\ldots X_{\sigma(k)}\right]=\iloskal{X_{\sigma(1)}\ldots X_{\sigma(k)}\phi}{\phi}
\end{equation}
with $k\in\nat$ and $\sigma$ mapping $\{1,\ldots,k\}$ into $\{1,\ldots,d\}$, we will always put $Z_j=X_{\sigma(j)}$ for $j=1,\ldots,k$, so \eqref{moment} will transform into $\iloskal{Z_1\ldots Z_k\phi}{\phi}$.

Next, for the sake of brevity, we introduce the following notation:
\[
i:=Z_i=X_{\sigma(i)}, i^-:=\amx{\sigma(i)}, i^0:=\azx{\sigma(i)}, i^+:=\apx{\sigma(i)};
\]
if $A$ and $B$ are any operators, then $\is{A}$ will stand for $\iloskal{A\phi}{\phi}$, while the composition $AB$ understood as $(AB)x=A(Bx)$ will be denoted using comma, i.e. $A,B:=AB$. For example,
\[
\is{1^-,2^-,\qc{3^-}{4},5,6}:=\iloskal{\amx{\sigma(1)}\amx{\sigma(2)}\qc{\amx{\sigma(3)}}{X_{\sigma(4)}}X_{\sigma(5)}X_{\sigma(6)}\phi}{\phi}.
\]


\subsubsection{Proof of Theorem \ref{t:characterization}}
Essentially the proof will follow the same line of reasoning as the proof of the main result of Stan and Whitaker \cite{stanwhitaker}. However, at some points we will need to take additional care of the issues arising from the lack of commutativity, and from the fact of using $q$-commutation relations.

The following Lemma generalizes "the product rule for commutators" \cite[Lemma 3.2]{stanwhitaker} to $q$-commutators:

\begin{lemma}\label{l:qprodrule}
For any operators $Y,U_1,\ldots,U_n$
\[
\qc{Y}{U_1\ldots U_n}=\sum_{i=0}^{n-1} q^i\cdot U_1\ldots U_i\qc{Y}{U_{i+1}}U_{i+2}\ldots U_n + (q^n-q)\cdot U_1\ldots U_nY.
\]
\end{lemma}

\begin{proof}
We proceed by induction on $n$. It is easy to check that for any operators $Y,V_1,V_2$
\begin{equation}\label{7}
\qc{Y}{V_1V_2}=\qc{Y}{V_1}V_2+q\cdot V_1\qc{Y}{V_2}+(q^2-q)\cdot V_1V_2Y,
\end{equation}
so the result holds for $n=2$. Assume it holds for some $n\ge2$. Putting $V_1=U_1\ldots U_n$ and $V_2=U_{n+1}$ in \eqref{7} we get
\begin{equation*}
\qc{Y}{U_1\ldots U_{n+1}}=\qc{Y}{U_1\ldots U_n}U_{n+1}+q\cdot U_1\ldots U_n\qc{Y}{U_{n+1}}+(q^2-q)\cdot U_1\ldots U_{n+1}Y.
\end{equation*}
By the induction assumption
\begin{multline*}
\qc{Y}{U_1\ldots U_{n+1}}=\sum_{i=0}^{n-1} q^i\cdot U_1\ldots U_i\qc{Y}{U_{i+1}}U_{i+2}\ldots U_{n+1}+
\\(q^n-q)\cdot U_1\ldots U_nYU_{n+1}+
q\cdot U_1\ldots U_n\qc{Y}{U_{n+1}}+(q^2-q)\cdot U_1\ldots U_{n+1}Y.
\end{multline*}
Since
$
U_1\ldots U_n Y U_{n+1}=U_1\ldots U_n\qc{Y}{U_{n+1}}+q\cdot U_1\ldots U_{n+1}Y,
$
we obtain
\[
\qc{Y}{U_1\ldots U_{n+1}}=\sum_{i=1}^n q^i\cdot U_1\ldots U_i\qc{Y}{U_{i+1}}U_{i+2}\ldots U_{n+1} + (q^{n+1}-q)\cdot U_1\ldots U_{n+1}Y,
\]
and the proof is complete.
\end{proof}

\begin{proof}[Proof of Theorem \ref{t:characterization}]
Assume first that $X$ and $Y$ are moment equal. It is not difficult to observe that the creation, preservation and annihilation operators associated with a random vector $U$ are defined in terms of the orthogonal projections of the chaos space $\mathcal{H}$ generated by $U_1,\ldots,U_d$ on the relevant homogenous chaos spaces $G_n$ of order $n$ (consult \cite{accakuostan1} for a detailed treatment of fully analogous commutative case). The projections in turn, from the Gram-Schmidt orthogonalizing procedure, are determined uniquely by the mixed moments of the vector $U$. Thus, from the fact that $X$ and $Y$ are moment equal it follows that their creation, preservation and annihilation operators are identical, so \eqref{anikrea} and \eqref{anizacho} hold.

Assume now \eqref{anikrea}, \eqref{anizacho} and \eqref{equimom}. Fix $k\in\nat$ and $\sigma(1),\ldots,\sigma(k)\in\{1,\ldots,d\}$. We will prove that
\[
\wo\left[X_{\sigma(1)}\ldots X_{\sigma(k)}\right]=\wo^\prime\left[Y_{\sigma(1)}\ldots Y_{\sigma(k)}\right]
\]
by showing that the expectation $\wo[X_{\sigma(1)}\ldots X_{\sigma(k)}]$ (and $\wo^\prime[Y_{\sigma(1)}\ldots Y_{\sigma(k)}]$) can be expressed in terms of the expectations of monomials of lower degrees, and that the monomials depend only on the $q$-commutators between the annihilation and creation operators, and the $q$-commutators between the annihilation and preservation operators. (Thus, by \eqref{anikrea} and \eqref{anizacho} the monomials will be identical when determined for $X$ using $\wo$ and for $Y$ using $\wo^\prime$.) Repeating the process of degree lowering, we will eventually arrive at the expectations $\wo[X_i]$ and $\wo^\prime[Y_i]$ which are equal by \eqref{equimom}. This will end the inductive reasoning.

By \eqref{xexpans}, \eqref{sprzez} and \eqref{a0},
\begin{multline*}
\wo\left[X_{\sigma(1)}\ldots X_{\sigma(k)}\right]=
\iloskal{Z_1\ldots Z_k\phi}{\phi}=\\
\iloskal{\amx{\sigma(1)}Z_2\ldots Z_k\phi}{\phi}+\iloskal{\azx{\sigma(1)} Z_2\ldots Z_k\phi}{\phi}+\iloskal{\apx{\sigma(1)}Z_2\ldots Z_k}{\phi}=\\
\iloskal{\amx{\sigma(1)}Z_2\ldots Z_k\phi}{\phi}+\iloskal{ Z_2\ldots Z_k\phi}{\azx{\sigma(1)}\phi}+\iloskal{Z_2\ldots Z_k}{\amx{\sigma(1)}\phi}=\\
\iloskal{\amx{\sigma(1)}Z_2\ldots Z_k\phi}{\phi}+\wo\left[X_{\sigma(1)}\right]\cdot\iloskal{ Z_2\ldots Z_k\phi}{\phi}.
\end{multline*}
Since $\iloskal{Z_2\ldots Z_k\phi}{\phi}$ is the expectation of a monomial of degree $k-1$ (in fact, the expectation of monomial $X_{\sigma(2)}\ldots X_{\sigma(k)}$), we will focus on $\iloskal{\amx{\sigma(1)}Z_2\ldots Z_k\phi}{\phi}$.

From \eqref{a0} and Lemma \ref{l:qprodrule} it follows that
\begin{multline*}
\iloskal{\amx{\sigma(1)}Z_2\ldots Z_k\phi}{\phi}=\is{1^-,2,\ldots,k}=\is{\qc{1^-}{2,\ldots,k}}=\\
\sum_{i=2}^{k} q^{i-2}\cdot\is{2,\ldots,i-1,\qc{1^-}{i},i+1,\ldots,k}+(q^{k-2}-q)\cdot\is{2,\ldots,k,1^-}.
\end{multline*}
Observe that the last term vanishes by another application of \eqref{a0}. Since the operator $\qc{1^-}{i}=\qc{\amx{\sigma(1)}}{X_{\sigma(i)}}$ maps $F_{k-i}$ into $F_{k-i}$, the expression
\begin{equation}\label{wyrazenie}
\is{2,\ldots,i-1,\qc{1^-}{i},i+1,\ldots,k}
\end{equation}
is the expectation of a polynomial of degree $k-2$ (for all $i=2,\ldots,k$). Thus
we have expressed $\wo[X_{\sigma(1)}\ldots X_{\sigma(k)}]$ in terms of expectations of some monomials of lower degrees.

What is left to show is that the monomials of lower degrees hidden in \eqref{wyrazenie} depend only on the $q$-commutators $\qc{\amx{\cdot}}{\apx{\cdot}}$ and $\qc{\amx{\cdot}}{\azx{\cdot}}$ from assumptions \eqref{anikrea} and \eqref{anizacho}. To this end, we first expand $\qc{1^-}{i}$ via \eqref{xexpans}:
\[
\qc{1^-}{i}=\qc{\amx{\sigma(1)}}{X_{\sigma(i)}}=\qc{\amx{\sigma(1)}}{\amx{\sigma(i)}+\azx{\sigma(i)}+\apx{\sigma(i)}},
\]
so $\iloskal{\amx{\sigma(1)}Z_2\ldots Z_k\phi}{\phi}$ becomes the sum of three terms of the form
\begin{equation}\label{sumazepsilonem}
\sum_{i=2}^{k} q^{i-2}\cdot\is{2,\ldots,i-1,\qc{\amx{\sigma(1)}}{a_X^\varepsilon(\sigma(i))},i+1,\ldots,k}
\end{equation}
with $\varepsilon\in\{-,0,+\}$. Of course, the choices $\varepsilon=0$ and $\varepsilon=+$ do not require any further arguments, so we focus on the case $\varepsilon=-$. (The point is that we do not assume that $\qc{a_X^-(\cdot)}{a_X^-(\cdot)}=\qc{a_Y^-(\cdot)}{a_Y^-(\cdot)}$.) We have
\begin{multline}\label{6}
\sum_{i=2}^{k} q^{i-2}\cdot\is{2,\ldots,i-1,\qc{\amx{\sigma(1)}}{\amx{\sigma(i)}},i+1,\ldots,k}=\\
\sum_{i=2}^{k} q^{i-2}\cdot\is{2,\ldots,i-1,1^-,i^-,i+1,\ldots,k}-
\sum_{i=2}^{k} q^{i-1}\cdot\is{2,\ldots,i-1,i^-,1^-,i+1,\ldots,k}.
\end{multline}
Consider the first sum after the equality sign in \eqref{6}. Observe that there are $k-i$ operators $Z$ that follow the rightmost annihilator in the product $2,\ldots,i-1,1^-,i^-,i+1,\ldots,k$. Lemma \ref{l:qprodrule} and formula \eqref{a0} imply that
\begin{multline*}
\sum_{i=2}^{k} q^{i-2}\cdot\is{2,\ldots,i-1,1^-,i^-,i+1,\ldots,k}=\\
\sum_{i=2}^{k-1} q^{i-2}\cdot\is{2,\ldots,i-1,1^-,\qc{i^-}{i+1,\ldots,k}}=\\
\sum_{i=2}^{k-1} q^{i-2}\sum_{j=i+1}^{k} q^{j-i-1}\cdot\is{2,\ldots,i-1,1^-,i+1,\ldots,j-1,\qc{i^-}{j},j+1,\ldots,k}.
\end{multline*}
We use \eqref{xexpans} again, and we again are interested only in
\begin{multline*}
\is{2,\ldots,i-1,1^-,i+1,\ldots,j-1,\qc{i^-}{j^-},j+1,\ldots,k}=\\
\is{2,\ldots,i-1,1^-,i+1,\ldots,j-1,i^-,j^-,j+1,\ldots,k}-\\q\cdot\is{2,\ldots,i-1,1^-,i+1,\ldots,j-1,j^-,i^-,j+1,\ldots,k}.
\end{multline*}
The number of the operators $Z$ that follow the rightmost annihilator in the product $2,\ldots,i-1,1^-,i+1,\ldots,j-1,i^-,j^-,j+1,\ldots,k$ (and $2,\ldots,i-1,1^-,i+1,\ldots,j-1,j^-,i^-,j+1,\ldots,k$) equals $k-j$, so it is strictly smaller than $k-i$ for all $j=i+1,\ldots,k$. Therefore repeating the above steps we will finally arrive at some products with the rightmost terms being some annihilators - the expectations of the products will vanish because of \eqref{a0}.

Clearly, the same reasoning can be applied to the second sum on the right-hand side of \eqref{6}.

Thus we have shown that the degree reduction procedure depends only on the $q$-commutators between the annihilation and creation operators, and the $q$-commutators between the annihilation and preservation operators, completing the induction.
\end{proof}


\section{An example}

Let $\left(\mathcal{A},\phi\right)$ be a probability space with the expectation $\wo$ supported by a Hilbert space $H$, and let $q\in[-1,1]$. In order to present the algorithm described in the proof of Theorem \ref{t:characterization} in action, we will solve the following $q$-commutator problem:

\begin{problem}\label{t:qnormal}
Compute the moments of random vector $X=(X_1,\ldots,X_d)$ if $X_i\in\mathcal{A}$ are such that
\begin{eqnarray*}
\left[\am{i};\ap{j}\right]_{q}&=&c(i,j)\cdot I,\\
\left[\am{i};\az{j}\right]_{q}&=&0,\\
\wo[X_i]&=&0,
\end{eqnarray*}
for $i,j=1,\ldots,d$, and the matrix $(c(i,j))_{i,j=1}^d$ is positive definite.
\end{problem}

Some special cases of Problem \ref{t:qnormal} are discussed in Subsection \ref{s:concluding}.

\subsection{A solution to Problem \ref{t:qnormal}}\label{ss:qnormal}

\subsubsection{Feynman diagrams} We now recall some basic facts on Feynman diagrams that are needed to solve Problem \ref{t:qnormal}; the notation we use is almost identical with the one from \cite[Section 2]{effrospopa}, which should be consulted for more information.

If $S$ is a linearly ordered set with an even number of elements, then a \emph{complete Feynman diagram $\gamma$ on $S$} is a partition of $S$ into two-element sets. We shall denote the collection of all such diagrams by $\FFF_c(S)$ and regard $\gamma$ as a set of ordered pairs
\begin{equation}\label{feynman}
\gamma=\{(i_1,j_1),\ldots,(i_n,j_n)\}
\end{equation}
if $S=\{i_1,\ldots,i_n,j_1,\ldots,j_n\}$, $i_1<\ldots<i_n$, $i_k<j_k$, $j_k\ne j_l$ for $k\ne l$ (the $j_k$ are generally unordered). We may represent a Feynman diagram with a simple graph, for example:
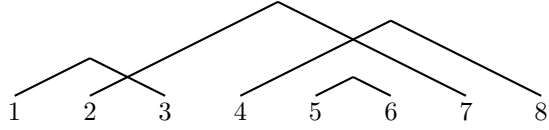
\begin{center}
\begin{pspicture}(0,0)(8,2)
\uput[u](.5,0){1}
\uput[u](1.5,0){2}
\uput[u](2.5,0){3}
\uput[u](3.5,0){4}
\uput[u](4.5,0){5}
\uput[u](5.5,0){6}
\uput[u](6.5,0){7}
\uput[u](7.5,0){8}

\psline(.5,.5)(1.5,1)
\psline(2.5,.5)(1.5,1)
\psline(4.5,.5)(5,0.75)
\psline(5.5,.5)(5,0.75)
\psline(3.5,.5)(5.5,1.5)
\psline(7.5,.5)(5.5,1.5)
\psline(1.5,.5)(4,1.75)
\psline(6.5,.5)(4,1.75)
\end{pspicture}
\end{center}
(this one corresponds to the Feynman diagram $\gamma=\{(1,3),(2,7),(4,8),(5,6)\}$ on $\{1,\ldots,8\}$).

A pair $(k,l)\in\gamma$ is a \emph{left crossing} for $(i,j)\in\gamma$ if $k<i<l<j$ and $c_l(i,j)$ will stand for the number of such left crossings for $(i,j)$. The number
\[
c(\gamma)=\sum_{(i,j)\in\gamma} c_l(i,j)
\]
will be referred to as the \emph{crossing number} of $\gamma$ (\emph{restricted crossing number} in Biane's \cite{biane} terminology). The crossing number can be determined by counting the intersection in the corresponding graph. For example, in the diagram above, $c_l(2,7)=c_l(4,8)=1$, $c_l(1,3)=c_l(5,6)=0$, and $c(\gamma)=2$.

For a given $2n$-element set $S$ and a Catalan sequence $\varepsilon\in C(2n)$ we call a complete Feynman diagram on $S$ of the form \eqref{feynman} compatible with $\varepsilon$ if $\varepsilon(i_k)=-1$ and $\varepsilon(j_k)=1$ for all $k$. If $A\subset C(2n)$ then we denote by $\FFF_c(S,A)$ the collection of all complete Feynman diagrams on $S$ compatible with some $\varepsilon\in A$.

\subsubsection{Some low-order moment computations}\label{sss:examples} To get a feeling of how the algorithm described in the proof of Theorem \ref{t:characterization} works under the assumptions of Problem \ref{t:qnormal}, we will compute a few low-order mixed moments of random variables $X_1,\ldots,X_d$.

The first moments $\wo[X_j]$ are assumed to be zero. Let us compute the ''covariances'':
\begin{multline*}
\wo[X_i X_j]=\iloskal{Z_1 Z_2\phi}{\phi}=\is{1,2}=\is{1^-,2}=\is{\qc{1^-}{2}}=\\
\is{\qc{1^-}{2^-}}+\is{\qc{1^-}{2^0}}+\is{\qc{1^-}{2^+}}.
\end{multline*}
The first term vanishes because $\left(\am{i}\am{j}-q\cdot\am{j}\am{i}\right)\phi=0$. From the assumptions of Problem \ref{t:qnormal} it follows that $\is{\qc{1^-}{2^0}}=0$ and $\is{\qc{1^-}{2^+}}=\is{\qc{\am{i}}{\ap{j}}}=c(i,j)\iloskal{I\phi}{\phi}$. Hence
$
\wo[X_i X_j]=c(i,j).
$

The third moments are zero. Indeed, Lemma \ref{l:qprodrule} implies that
\begin{equation*}
\iloskal{Z_1 Z_2 Z_3\phi}{\phi}=\is{1,2,3}=\is{1^-,2,3}=\is{\qc{1^-}{2,3}}=\is{\qc{1^-}{2},3}+q\cdot\is{2,\qc{1^-}{3}},
\end{equation*}
and from the assumptions of Problem \ref{t:qnormal} and formulas \eqref{moreminuses} and \eqref{a0}, it follows that both
\[
\is{\qc{1^-}{2},3}=c(\sigma(1),\sigma(2))\cdot\wo[X_{\sigma(3)}]+\is{1^-,2^-,3}-q\cdot\is{2^-,1^-,3},
\]
and
\[
\is{2,\qc{1^-}{3}}=c(\sigma(1),\sigma(3))\cdot\wo[X_{\sigma(2)}]+\is{2,1^-,3^-}-q\cdot\is{2,3^-,1^-}
\]
vanish.

The last example we present is a computation of the fourth mixed moments. We have
\begin{multline*}
\iloskal{Z_1\ldots Z_4 \phi}{\phi}=\is{1,2,3,4}=\is{1^-,2,3,4}=\is{\qc{1^-}{2,3,4}}=\\
\is{\qc{1^-}{2},3,4}+q\cdot\is{2,\qc{1^-}{3},4}+q^2\cdot\is{2,3,\qc{1^-}{4}}.
\end{multline*}
Now the key point is the following. It turns out that a smart way of handling the three elements of the sum above is to expand only the first one via \eqref{xexpans} (if one expands at each stage all three elements of relevant sums, the total number of elements under consideration grows fast and it is hard to find any pattern or regularity). Let us expand then:
\begin{multline*}
\is{\qc{1^-}{2},3,4}=\is{\qc{1^-}{2^-},3,4}+\is{\qc{1^-}{2^0},3,4}+\is{\qc{1^-}{2^+},3,4}=\\
\is{1^-,2^-,3,4}-q\cdot\is{2^-,1^-,3,4}+c(\sigma(1),\sigma(2))\cdot\is{3,4}.
\end{multline*}
The term $\is{3,4}=\is{3^-,4}$ is the expectation of a monomial of a lower degree (in fact it equals $c(\sigma(3),\sigma(4))$; we leave $\is{1^-,2^-,3,4}$ temporarily unexpanded and  focus on
\begin{multline*}
-q\cdot\is{2^-,1^-,3,4}=-q\cdot\is{2^-,\qc{1^-}{3,4}}=-q\cdot\is{2^-,\qc{1^-}{3},4}-q^2\cdot\is{2^-,3,\qc{1^-}{4}}.
\end{multline*}
Thus we arrive at
\begin{multline*}
\is{1,2,3,4}=c(\sigma(1),\sigma(2))\cdot c(\sigma(3),\sigma(4))+\is{1^-,2^-,3,4}+\\
q\cdot\left(\is{2,\qc{1^-}{3},4}-\is{2^-,\qc{1^-}{3},4}\right)+q^2\cdot\left(\is{2,3,\qc{1^-}{4}}-\is{2^-,3,\qc{1^-}{4}}\right).
\end{multline*}
By \eqref{xexpans} and \eqref{a0} applied to $Z_2$ in $\is{2,\qc{1^-}{3},4}$ and $\is{2,3,\qc{1^-}{4}}$
\begin{equation*}
\is{2,\qc{1^-}{3},4}=\is{2^-,\qc{1^-}{3},4},\ \is{2,3,\qc{1^-}{4}}=\is{2^-,3,\qc{1^-}{4}},
\end{equation*}
so we are left with a computation of $\is{1^-,2^-,3,4}$:
\begin{equation}\label{3}
\is{1^-,2^-,3,4}=\is{1^-,\qc{2^-}{3,4}}=\is{1^-,\qc{2^-}{3},4}+q\cdot\is{1^-,3,\qc{2^-}{4}}.
\end{equation}
We again expand only $Z_3$ in $\qc{2^-}{3}$ and obtain
\begin{multline*}
\is{1^-,2^-,3,4}=c(\sigma(2),\sigma(3))\cdot\is{1^-,4}+\is{1^-,2^-,3^-,4}-\\
q\cdot\is{1^-,3^-,2^-,4}+q\cdot\is{1^-,3,\qc{2^-}{4}}.
\end{multline*}
Formula \eqref{moreminuses} implies that $\is{1^-,2^-,3^-,4}=\is{1^-,3^-,2^-,4}=0$ and
\begin{equation*}
\is{1^-,3,\qc{2^-}{4}}=\is{1^-,3,\qc{2^-}{4^+}}=c(\sigma(2),\sigma(4))\cdot\is{1^-,3}.
\end{equation*}
Concluding, we arrived at
\begin{multline*}
\is{1,2,3,4}=c(\sigma(1),\sigma(2))\cdot c(\sigma(3),\sigma(4))+c(\sigma(2),\sigma(3))\cdot c(\sigma(1),\sigma(4))+\\
q\cdot c(\sigma(2),\sigma(4))\cdot c(\sigma(1),\sigma(3)).
\end{multline*}

\subsubsection{The solution}

Now, for the rest of Subsection \ref{ss:qnormal}, fix $k\in\nat$ and $\sigma(1),\ldots,\sigma(k)\in\{1,\ldots,d\}$. We will be considering $K=\{k_1,\ldots,k_{N}\}$ ($N$ is a positive integer), satisfying
\begin{equation}\label{assforK}
k_1<\ldots<k_{N}\textrm{ and }k_i\in\{1,\ldots,k\},\ i=1,\ldots,N.
\end{equation}
If $N$ is an even number (say $N=2n$ for some $n\in\nat$), then for any $K$ satisfying \eqref{assforK} and a Feynman diagram $\gamma=\{(i_1,j_1),\ldots,(i_n,j_n)\}$ on $K$, we define
\begin{equation*}
\nu(\gamma)=\prod_{m=1}^n c\left(\sigma(k_{i_m}),\sigma(k_{j_m})\right).
\end{equation*}

In order to solve Problem \ref{t:qnormal} we will compute the mixed moments of the form \eqref{moment}:
\begin{equation*}
\wo\left[X_{\sigma(1)}\ldots X_{\sigma(k)}\right]=\iloskal{X_{\sigma(1)}\ldots X_{\sigma(k)}\phi}{\phi}=
\iloskal{Z_1\ldots Z_{k}\phi}{\phi}=\is{1,\ldots,k}.
\end{equation*}
Our task will be to show that they match the ones of $q$-Gaussian variables that are given by

\begin{theorem}[$q$-Wick theorem]\label{t:q-Wick}
For any $q$-Gaussian random variables $Y_i$ with "covariances" $\wo[Y_i Y_j]=c(i,j)$, we have
\begin{equation}\label{q-Wick}
\wo[Y_{\sigma(1)}\ldots Y_{\sigma(2n)}]=\sum_{\gamma\in \FFF_c(\{1,\ldots,2n\})} q^{c(\gamma)}\cdot\nu(\gamma)
\end{equation}
where
\[
\nu(\gamma)=\nu(\{(i_1,j_1),\ldots,(i_n,j_n)\})=\prod_{m=1}^n c\left(\sigma(i_m),\sigma(j_m)\right),
\]
and $\wo[Y_{\sigma(1)}\ldots Y_{\sigma(2n+1)}]=0$.
\end{theorem}

Theorem \ref{t:q-Wick} was found by Bo\.{z}ejko and Speicher \cite{bosp} (see also \cite{effrospopa} for an alternative treatment). It generalizes the famous Wick formula (see Subsection \ref{s:concluding}) for computing higher-order moments of the multivariate normal distribution in terms of its covariance matrix, which had apparently been discovered first by Isserlis \cite{isserlis} and rediscovered independently by Wick \cite{wick}.

We begin with a lemma that captures the cancellation mechanism that has already appeared in the computation of the fourth moments in Subsection \ref{sss:examples}:

\begin{lemma}\label{l:mechanizmy} Fix $N\in\nat$ and $K=\{k_1,\ldots,k_{N}\}$ satisfying \eqref{assforK} and define for $1\le l<m\le N$
\[
R(l;m;U):=\is{k_1^-,\ldots,k_l^-,k_m,U}-\is{k_1^-,\ldots,k_l^-,k_m^-,U},
\]
$U$ is any operator.

Under the assumptions of Problem \ref{t:qnormal}, the following statements hold:
\begin{equation}\label{roznice}
R(l;m;U)=\sum_{j=1}^{l} q^{l-j}\cdot c\left(\sigma(k_j),\sigma(k_m)\right)\cdot \is{k_1^-,\ldots,\widehat{k_j^-},\ldots,k_l^-,U}
\end{equation}
(here and later the hat sign stands for the deletion of the relevant operator from the product),
\begin{multline}\label{mechogolny}
\is{k_1^-,\ldots,k_m^-,k_{m+1},\ldots,k_{N}}=\\
\sum_{j=1}^m q^{m-j}\cdot c(\sigma(k_j),\sigma(k_{m+1}))\cdot\is{k_1^-,\ldots,\widehat{k_j^-},\ldots,k_m^-,k_{m+2},\ldots,k_{N}}+\\
\is{k_1^-,\ldots,k_{m+1}^-,k_{m+2},\ldots,k_{N}}.
\end{multline}
\end{lemma}

\begin{proof}
We start with a justification of the first statement. By Lemma \ref{l:qprodrule}
\begin{multline*}
R(l;m;U)=\is{k_1^-,\ldots,k_l^-,k_m,U}-\is{k_1^-,\ldots,k_l^-,k_m^-,U}=\\
\is{k_1^-,\ldots,k_{l-1}^-,\qc{k_l^-}{k_m},U}-\is{k_1^-,\ldots,k_{l-1}^-,\qc{k_l^-}{k_m^-},U}\\
+q\cdot\left\{
\is{k_1^-,\ldots,k_{l-1}^-,k_m,\qc{k_l^-}{U}}-\is{k_1^-,\ldots,k_{l-1}^-,k_m^-,\qc{k_l^-}{U}}\right\}
\end{multline*}
From \eqref{xexpans} it follows that the first difference equals
$$
c(\sigma(k_l),\sigma(k_m))\cdot\is{k_1^-,\ldots,k_{l-1}^-,U};
$$
the difference in the curly bracket is $R(l-1;m;\qc{k_l^-}{U})$, which in fact equals $R(l-1;m;k_l^-,U)$. Hence
\begin{equation}\label{1}
R(l;m;U)=c(\sigma(k_l),\sigma(k_m))\cdot\is{k_1^-,\ldots,k_{l-1}^-,U}+q\cdot R(l-1;m;k_l^-,U).
\end{equation}
On the other hand
\begin{multline*}
R(1;m;U)=\is{k_1^-,k_m,U}-\is{k_1^-,k_m^-,U}=c(\sigma(k_1),\sigma(k_m))\cdot\is{U}+\\q\cdot\left\{\is{k_m,\qc{k_1^-}{U}}-\is{k_m^-,\qc{k_1^-}{U}}\right\}.
\end{multline*}
By \eqref{xexpans}, \eqref{sprzez} and \eqref{a0}, the difference in the curly brackets vanishes, so $R(1;m;U)=c(\sigma(k_1),\sigma(k_m)\cdot\is{U}$. Combining \eqref{1} with this, we get
\begin{equation*}
R(l;m;U)=\sum_{j=0}^{l-1} q^j\cdot c(\sigma(k_{l-j}),\sigma(k_m))\cdot \is{k_1^-,\ldots,\widehat{{k_{l-j}^-}},\ldots,k_l^-,U}.
\end{equation*}
A simple reordering gives \eqref{roznice}.

Formula \eqref{mechogolny} readily follows from the definition of $R(m;m+1;k_{n+2},\ldots,k_N)$ and \eqref{roznice}.

%
\end{proof}

\begin{corollary}\label{c:nieparzyste}
If $n\in\nat$ and $K=\{k_1,\ldots,k_{2n-1}\}$ satisfies \eqref{assforK}, then for all $m\in\{1,\ldots,2n-1\}$
\[
\is{k_1^-,\ldots,k_m^-,k_{m+1},\ldots,k_{2n-1}}=0.
\]
\end{corollary}
\begin{proof}
It immediately follows from \eqref{mechogolny} and the assumption $\wo[X_j]=0$, $j=1,\ldots,d$, combined with \eqref{moreminuses}.
\end{proof}

\begin{corollary}
For any $m,n\in\nat$, $m\le n$ and $K=\{k_1,\ldots,k_{2n}\}$ satisfying \eqref{assforK},
\begin{multline}\label{mechogolnykoncowy}
\is{k_1^-,\ldots,k_m^-,k_{m+1},\ldots,k_{2n}}=\\
\sum_{l=m+1}^{n+1}\sum_{j=1}^{l-1} q^{l-1-j}\cdot c(\sigma(k_j),\sigma(k_l))\cdot
\is{k_1^-,\ldots,\widehat{k_j^-},\ldots,k_{l-1}^-,k_{l+1},\ldots,k_{2n}}.
\end{multline}
\end{corollary}
\begin{proof}
If $m<n$ then iterating \eqref{mechogolny} one gets
\begin{multline*}
\is{k_1^-,\ldots,k_m^-,k_{m+1},\ldots,k_{2n}}=\\
\sum_{l=m+1}^{n}\sum_{j=1}^{l-1} q^{l-1-j}\cdot c(\sigma(k_j),\sigma(k_l))\cdot
\is{k_1^-,\ldots,\widehat{k_j^-},\ldots,k_{l-1}^-,k_{l+1},\ldots,k_{2n}}+\\
\is{k_1^-,\ldots,k_n^-,k_{n+1},\ldots,k_{2n}}.
\end{multline*}
On the other hand, from \eqref{mechogolny} and \eqref{moreminuses} it follows that
\begin{multline}\label{mechkoncowy}
\is{k_1^-,\ldots,k_n^-,k_{n+1},\ldots,k_{2n}}=\\
\sum_{j=1}^n q^{n-j}\cdot c(\sigma(k_j),\sigma(k_{n+1}))\cdot\is{k_1^-,\ldots,\widehat{k_j^-},\ldots,k_n^-,k_{n+2},\ldots,k_{2n}}.
\end{multline}
\end{proof}

As it was seen in the examples of low-order moment computations from subsection \ref{sss:examples}, the algorithm from Theorem \ref{t:characterization} applied to the expressions of the form
\begin{equation*}
\is{k_1^-,\ldots,k_m^-,{k_{m+1}},\ldots,k_{2n}}
\end{equation*}
expands them, i.e. it transforms the operators $Z_{k_{m+1}},\ldots,Z_{k_{2n}}$ either into the relevant annihilators, or creators (the expressions in which any of $Z$'s is transformed into the preservation operator, vanish due to the assumption $\qc{\am{\cdot}}{\az{\cdot}}=0$). Proposition \ref{p:maincombi} stated below links the expressions
with the Feynman diagrams: the annihilators in the expanded versions are represented by the first elements of the ordered pairs from \eqref{feynman} (the ascending lines in the corresponding graph), while the creators are represented by the second elements (the descending lines in the graph).

\begin{proposition}\label{p:maincombi}
For any $m,n\in\nat$, $m\le n$ and $K=\{k_1,\ldots,k_{2n}\}$ satisfying \eqref{assforK},
\begin{equation}\label{nadiagramach}
\is{k_1^-,\ldots,k_m^-,{k_{m+1}},\ldots,k_{2n}}=\sum_{\gamma\in \FFF_c(K,A_{m,2n})} q^{c(\gamma)}\cdot\nu(\gamma),
\end{equation}
where
$
A_{i,j}=\{\varepsilon=(\varepsilon(1),\ldots,\varepsilon(j))\in C(j):\varepsilon(1)=\ldots=\varepsilon(i)=-1\}.
$
\end{proposition}

\begin{proof}
We proceed by induction on $n$. Clearly, \eqref{nadiagramach} holds for $n=1$. Assume then it holds for some $n-1\in\nat$ and all $m\le n-1$.

Now, let $m\in\{1,\ldots,n\}$. By the induction assumption applied to \eqref{mechogolnykoncowy},
\begin{multline*}
\is{k_1^-,\ldots,k_m^-,k_{m+1},\ldots,k_{2n}}=\\
\sum_{l=m+1}^{n+1}\sum_{j=1}^{l-1} q^{l-1-j}\cdot c(\sigma(k_j),\sigma(k_l))\cdot
\sum_{\gamma\in\FFF_c(K\setminus\{k_j,k_l\},A_{l-2,2n-2})} q^{c(\gamma)}\cdot\nu(\gamma).
\end{multline*}

For $l\in\{m+1,\ldots,n+1\}$ and $j\in\{1,\ldots,l-1\}$, consider any Feynman diagram $\gamma$ from $\FFF_c(K\setminus\{k_j,k_l\},A_{l-1,2n-2})$. If one extends $\gamma$ to the unique Feynman diagram $\gamma^\prime\in\FFF_c(K,A_{l-1,2n})$ by putting $k_j$ and $k_l$ into their natural places and linking them together, then
$\nu(\gamma^\prime)= \nu(\gamma)\cdot c(\sigma(k_j),\sigma(k_l))$ and $c(\gamma^\prime)=c(\gamma)+l-1-j$, as there are $l-1-j$ elements of $K$ between $k_j$ and $k_l$ and none of the elements that are between $k_j$ and $k_l$ is linked with any other from these elements. (That is, there do not exist $k_\alpha<k_\beta$, both from $\{k_{j+1},\ldots,k_{l-1}\}$, such that $(k_\alpha,k_\beta)\in\gamma^\prime$ -- if they existed, one would not have $\gamma^\prime\in\FFF_c(K,A_{l-1,2n})$ because of $\varepsilon(\beta)=1$. So ascending lines that correspond to all the $l-1-j$ elements cross the lines joining $k_j$ and $k_l$.) Hence
\begin{equation*}
\is{k_1^-,\ldots,k_m^-,k_{m+1},\ldots,k_{2n}}=\sum_{l=m+1}^{n+1}\sum_{j=1}^{l-1}\sum_{\gamma^\prime} q^{c(\gamma^\prime)}\cdot\nu(\gamma^\prime)
\end{equation*}
with $\gamma^\prime$'s running through the set of all complete Feynman diagrams on $K$ with $k_j$ and $k_l$ linked together, i.e. such that $(k_j,k_l)\in\gamma$.


Since all the $\gamma^\prime$'s from the formula above are in fact all the Feynman diagrams on $K$ compatible with $A_{m,2n}$, we get \eqref{nadiagramach} with $m\in\{1,\ldots,n\}$.

\end{proof}

\begin{proof}[Solution of Problem \ref{t:qnormal}]
If $k$ is an odd number then
\begin{equation*}
\wo\left[X_{\sigma(1)}\ldots X_{\sigma(k)}\right]=\is{1,\ldots,k}=\is{1^-,2,\ldots,k}=0
\end{equation*}
by Corollary \ref{c:nieparzyste}. If $k=2n$ for some $n\in\nat$ then
\begin{equation*}
\wo\left[X_{\sigma(1)}\ldots X_{\sigma(k)}\right]=\is{1,\ldots,2n}=\is{1^-,2,\ldots,2n}=\sum_{\gamma\in\FFF_c(\{1,\ldots,2n\},A_{1,2n}) } q^{c(\gamma)}\cdot\nu(\gamma) ,
\end{equation*}
the second equality follows from \eqref{sprzez} and the assumptions of the Theorem, and the third one is a consequence of Proposition \ref{p:maincombi}. Since $$\FFF_c(\{1,\ldots,2n\},A_{1,2n})=\FFF_c(\{1,\ldots,2n\},C(2n))=\FFF_c(\{1,\ldots,2n\}),$$ we finally arrive at the following formula:
\begin{equation*}
\wo\left[X_{\sigma(1)}\ldots X_{\sigma(k)}\right]=
\begin{cases}
0& \textrm{if } k\textrm{ is an odd number,}\\
\sum_{\gamma\in\FFF_c(\{1,\ldots,2n\}) } q^{c(\gamma)}\cdot\nu(\gamma)  & \textrm{if } k=2n.
\end{cases}
\end{equation*}
\end{proof}

Thus, by Theorem \ref{t:q-Wick}, we get the following

\begin{corollary}
Random variables $X_1,\ldots,X_d$ from Problem \ref{t:qnormal} are moment equal to $q$-Gaussian random variables. In fact, by Theorem \ref{t:characterization}, the assumptions of Problem \ref{t:qnormal} characterize random variables that are moment equal to $q$-Gaussians.
\end{corollary}

\subsection{Concluding remarks}\label{s:concluding}

Problem \ref{t:qnormal} covers some important distributions as its special cases.

If $q=1$ then we have the usual commutator $[A,B]=AB-BA$ and $q$-Gaussian random variables may be identified with jointly Gaussian random variables. Unsurprisingly, from the solution of Problem \ref{t:qnormal} we get the following

\begin{corollary}\label{c:multinormal}
If $X_i\in\mathcal{A}$ are such that for all $i,j\in\{1,\ldots,d\}$
\begin{eqnarray*}
\left[\am{i};\ap{j}\right]&=&c(i,j)\cdot I,\\
\left[\am{i};\az{j}\right]&=&0,\\
\wo[X_i]&=&0,
\end{eqnarray*}
then
\begin{equation*}
\wo\left[X_{\sigma(1)}\ldots X_{\sigma(2n)}\right]=\sum_{\gamma\in\FFF_c(\{1,\ldots,2n\})}\nu(\gamma),
\end{equation*}
and $\wo\left[X_{\sigma(1)}\ldots X_{\sigma(2n+1)}\right]=0$.
\end{corollary}
This is the celebrated Wick's formula, which is a special case of Theorem \ref{t:q-Wick}. Since Gaussian distributions are determined by moments, the assumptions of Corollary \ref{c:multinormal} characterize zero mean $d$-dimensional normal vectors (in particular, when $d=1$ we get the solution of Example 4.1 from \cite{stanwhitaker}).

Another interesting case is $q=0$ with $[A,B]_0=AB$, and with such convention in the $q$-Wick Theorem that $q^c=0$ for $c\ne0$, and $q^0=1$. A Feynman diagram $\gamma$ is \emph{noncrossing} if $c(\gamma)=0$. If $\NC(S)$ denotes the noncrossing diagrams on a set $S$, then $0$-version of the $q$-Wick formula reads: if $k=2n$ then
\begin{equation}\label{0Wick}
\wo\left[X_{\sigma(1)}\ldots X_{\sigma(k)}\right]=\sum_{\gamma\in\NC(\{1,\ldots,2n\})} \nu(\gamma);
\end{equation}
if $k$ is odd then the expectation vanishes.
Since \eqref{0Wick} defines\emph{ a semicircular family of covariance} $(c(i,j))_{i,j=1,\ldots,d}$ (see \cite[Definition 8.15]{nicaspeicher}), the solution of Problem \ref{t:qnormal} gives the following result:
\begin{corollary}
If $X_i\in\mathcal{A}$ are such that for all $i,j\in\{1,\ldots,d\}$
\begin{eqnarray*}
\left[\am{i};\ap{j}\right]_{0}&=&c(i,j)\cdot I,\\
\left[\am{i};\az{j}\right]_{0}&=&0,\\
\wo[X_i]&=&0,
\end{eqnarray*}
then they are moment equal to a semicircular family of covariance $(c(i,j))_{i,j=1,\ldots,d}$.
\end{corollary}
By Theorem \ref{t:characterization}, both Corollaries are in fact characterizations. When $q=-1$ one gets a similar characterization of the Fermi field (which leads to $\pm1$-valued symmetric random variables in classical probability).

\subsection*{Acknowledgement} The authors thank Jacek Weso{\l}owski for encouragement and discussions, and Pawe{\l} Naroski for technical help. We have benefited from remarks of Kamil Szpojankowski. A part of the material presented in this paper formed the first author's Master Thesis in Politechnika Warszawska.

\bibliographystyle{amsplain}
\bibliography{c:/w/wm/res/var/toolbox/WM}

\end{document}